\documentclass[12pt,reqno]{amsart}
\usepackage{amsfonts}

\usepackage{verbatim,amscd,amsmath,amsthm,amstext,amssymb,amsfonts,latexsym,lscape,rawfonts}
\usepackage[arrow,matrix,curve]{xy}

\newtheorem{theorem}{Theorem}[section]
\newtheorem{corollary}[theorem]{Corollary}
\newtheorem{lemma}[theorem]{Lemma}
\newtheorem{proposition}[theorem]{Proposition}
\theoremstyle{definition}
\newtheorem{definition}{Definition}

\newtheorem{remark}{Remark}
\newtheorem*{acknow}{Acknowledgments}

\newcommand{\R}{\mathbb{R}}
\newcommand{\C}{\mathbb{C}}

\newcommand{\St}{\mathbb{S}}
\newcommand{\F}{\mathbb{F}}
\newcommand{\Z}{\mathbb{Z}}

\def\sn{\mathrm{sn}}
\def\cn{\mathrm{cn}}
\def\dn{\mathrm{dn}}

\begin{document}

\title[A new look at equivariant minimal Lagrangian surfaces]
{A new look at equivariant minimal Lagrangian surfaces in $\mathbb{C}P^2$}
\author{Josef F.~Dorfmeister}
\address{Fakult\"{a}t F\"{u}r Mathematik, TU-M\"{u}nchen, Boltzmann Str. 3,
D-85747, Garching, Germany}
\email{dorfm@ma.tum.de}
\author{Hui Ma}
\address{Department of Mathematical Sciences, Tsinghua University,
Beijing 100084, P.R. China} \email{hma@math.tsinghua.edu.cn}

\maketitle

\begin{abstract} 
In this note, we present a new look at translationally equivariant minimal Lagrangian surfaces in the complex projective plane via the loop group method.
\end{abstract}

\section{Introduction}
\label{sec:Intro}

Minimal Lagrangian surfaces in the complex projective plane $\mathbb{C}P^2$ endowed with the Fubini-Study metric are of great interest from the point of view of differential geometry, symplectic geometry and mathematical physics. 
(\cite{CU94, Sharipov, MM, Haskins_Am04, McIn, Haskins_Inv04}). 
They give rise to  local models of singular special Lagrangian $3$-folds in Calabi-Yau 3-folds, hence play an important role in the development of mirror symmetry (\cite{Joyce}). The Gauss-Codazzi equations for minimal Lagrangian surfaces in $\mathbb{C}P^2$
are given by
\begin{equation*}
u_{z\bar z}=e^{-2u}|\psi|^2-e^{u},\quad 
\psi_{\bar z}=0,
\end{equation*}
where $g=2e^{u}dzd\bar{z}$ is the Riemannian metric of a Riemann surface and $\psi dz^3$ is a holomorphic cubic differential defined on the surface. 
Since any minimal Lagrangian surface of genus zero in $\mathbb{C}P^2$ is totally geodesic, it is the standard  immersion of $S^2$ in $\mathbb{C}P^2$ (\cite{Yau1974, Naitoh-Takeuchi82}).  In a nice paper \cite{CU94} by Castro and Urbano, they reduced the PDE above to an ODE and constructed translationally equivariant minimal Lagrangian tori in $\C P^2$. 
Later on it was shown that any minimal Lagrangian immersed surface of genus one in $\mathbb{C}P^2$ can be  constructed in terms of algebraically completely integrable systems (\cite{Sharipov, MM, McIn}). 
Recently, a loop group method introduced by Dorfmeister, Pedit and Wu (\cite{DPW}) has proven to be  efficient in constructing surfaces related to a family of flat connections with nontrivial topology. 
As a preparation for the construction of minimal Lagrangian surfaces with \lq\lq ends\rq\rq \, in $\C P^2$,   
we would like to present a new look at translationally equivariant minimal Lagrangian surfaces in $\C P^2$ via the loop group method in this note.

The paper is organized as follows: In Section \ref{sec:miL}, we recall the basic set-up for minimal Lagrangian surfaces in $\C P^2$. In Section \ref{Sec:Equiv}, we  explain the definition of equivariant minimal Lagrangian surfaces in $\C P^2$.  In Section \ref{Sec:trans equiv}, we show that every translationally equivariant minimal Lagrangian surface in $\mathbb{C}P^2$ is generated by a degree one constant potential.
In Section \ref{Sec:Iwasawa}, we present an explicit Iwasawa decomposition for any translationally equivariant minimal Lagrangian surface. In Section \ref{Sec:equiv cylinder tori}, we discuss the periodicity condition for translationally equivariant minimal Lagrangian cylinders and tori. Finally, we compare our loop group  approach to the work of  Castro-Urbano (\cite{CU94}).

\section{Minimal Lagrangian surfaces in $\mathbb{C}P^2$ }
\label{sec:miL}

We recall briefly the basic set-up for minimal Lagrangian surfaces in $\mathbb{C}P^2$. For details we refer to \cite{MM} and references therein.

Let $\mathbb CP^2$ be the complex projective plane endowed with the Fubini-Study metric of constant holomorphic
sectional curvature 4.
Let $f:M\rightarrow {\mathbb C}P^2$ be a Lagrangian immersion of an oriented
surface. The induced metric on $M$ generates a
conformal structure with
respect to which the metric is $g=2e^{u} dzd{\bar z}$, and 
where $z=x+iy$ is a local conformal coordinate on $M$ and $u$
is a real-valued function defined on $M$ locally.
For any Lagrangian immersion $f$, there exists a local horizontal lift $F: {\rm U}\rightarrow S^5(1) =\{Z\in \mathbb{C}^3 | \, Z\cdot \bar{Z}=1\}$,
where 
$Z\cdot \overline{W}=\sum_{k=1}^{3} z_{k} \overline{w_{k}}$
denotes the Hermitian inner product for any $Z=(z_1,z_2,z_3)$ and $W=(w_1,w_2,w_3)\in \mathbb{C}^3$. 
In fact, choose any local lift $F$. Then $dF\cdot \bar{F}$ is a closed one-form. Hence there exists a real function $\eta\in C^\infty ({\rm U})$ locally such that $i d\eta= dF\cdot \bar F$.
Then $\tilde{F}=e^{-i\eta} F$ is a local 
horizontal lift of $f$ to $S^5(1)$.
We can therefore assume
\begin{equation}\label{horizontal}
F_z \cdot {\overline F}=F_{\bar z}\cdot {\overline F} =0.
\end{equation}
 
The fact that the metric $g$ is conformal is equivalent to
\begin{equation}\label{conformal}
\begin{split}
&F_z\cdot \overline{F_z}=F_{\bar z}\cdot \overline{F_{\bar z}}=e^{u},\\
& F_z\cdot \overline{F_{\bar z}}=0.
\end{split}
\end{equation}
Thus 
$\mathcal{F}=(e^{-\frac{u}{2}}F_z, e^{-\frac{u}{2}}F_{\bar z}, F)$ is a Hermitian orthonormal moving frame 
globally defined on the universal cover of $M$.
Furthermore, let us assume that $f$ is minimal now.
It follows from \eqref{horizontal} and \eqref{conformal} and the minimality of 
 $f$ that $\mathcal{F}$ satisfies the frame equations
\begin{equation}\label{eq:frame1}
\mathcal{F}_{z}=\mathcal{F} {\mathcal U}, \quad \mathcal{F}_{\bar z}=\mathcal{F} {\mathcal V},
\end{equation}
where
\begin{equation}\label{eq:UV1}
{\mathcal U}=\left(\begin{array}{ccc}
                   \frac{u_z}{2} & 0 & e^{\frac{u}{2}} \\
      e^{-u}\psi  &-\frac{u_z}{2}  & 0 \\
       0 & -e^{\frac{u}{2}}  &0 \\
     \end{array}
   \right),
   \quad
{\mathcal V}=\left(
     \begin{array}{ccc}
      -\frac{u_{\bar z}}{2}  &  - e^{-u}\bar\psi  &0 \\
       0& \frac{u_{\bar z}}{2} &e^{\frac{u}{2}} \\
       -e^{\frac{u}{2}}  &0 &0 \\
     \end{array}
   \right),
\end{equation}
with
\begin{equation}\label{eq:phipsi}
\psi=F_{zz}\cdot\overline{F_{\bar z}}.
\end{equation}
The cubic differential $\Psi=\psi dz^3$ is globally defined on $M$  and
independent of the choice of the local lift. The differential $\Psi$ is called the Hopf differential of $f$.

The compatibility condition of  the equations \eqref{eq:frame1} is
${\mathcal U}_{\bar z}-{\mathcal V}_z =[{\mathcal U},{\mathcal V}]$,  
and using \eqref{eq:UV1} this turns out to be equivalent to
 \begin{eqnarray}
u_{z\bar z}+e^u-e^{-2u}|\psi|^2&=0,\label{eq:mLsurfaces}\\
\psi_{\bar z}&=0.\label{eq:Codazzi}
\end{eqnarray}
Notice that the integrability conditions  \eqref{eq:mLsurfaces}-\eqref{eq:Codazzi} are invariant under the transformation
$\psi \rightarrow \nu \psi$ for any $\nu \in S^1$.

This implies that after replacing $\psi$ in \eqref{eq:UV1}   by  $\psi^\nu=\nu\psi$ the equations \eqref{eq:frame1} are still integrable.
Therefore, the solution $\mathcal{F}(z,\bar{z}, \nu)$  to this changed system is a frame of some minimal Lagrangian surface $f^\nu$.

It turns out to be convenient to consider in place of the frames $\mathcal{F}(z,\bar{z}, \nu)$ the gauged frames 
\begin{equation*}
\mathbb{F}(\lambda)=\mathcal{F}(\nu)\left(
                                   \begin{array}{ccc}
                                     -i\lambda & 0 & 0 \\
                                     0 &  \frac{1}{i\lambda} & 0 \\
                                     0 & 0 &  1 \\
                                   \end{array}
                                 \right),
\end{equation*}
where $i\lambda^3 \nu=1$.

For these frames we obtain the equations
\begin{equation}\label{eq:mathbbF}
\begin{split}
\mathbb{F}^{-1}\mathbb{F}_z&=
\frac{1}{\lambda}\left(
   \begin{array}{ccc}
     0 & 0 & i e^{\frac{u}{2}} \\
  -i \psi e^{-u}   & 0 & 0 \\
     0 &  i e^{\frac{u}{2}} & 0 \\
   \end{array}
 \right)+\left(
   \begin{array}{ccc}
   \frac{u_z}{2}  & &  \\
      & -\frac{u_z}{2} &  \\
      & & 0 \\
   \end{array}
 \right)\\
 &:=\lambda^{-1}U_{-1}+U_0,\\
\mathbb{F}^{-1}\mathbb{F}_{\bar{z}} &=\lambda \left(
   \begin{array}{ccc}
     0 &  -i\bar{\psi} e^{-u} & 0 \\
     0& 0   & i e^{\frac{u}{2}} \\
     i e^{\frac{u}{2}} &0 & 0 \\
   \end{array}
 \right)+\left(
   \begin{array}{ccc}
     -\frac{u_{\bar z}}{2} & &  \\
      & \frac{u_{\bar z}}{2}  & \\
     & & 0 \\
   \end{array}
 \right)\\
  &:=\lambda V_{1}+V_0.
\end{split}
\end{equation}

\begin{proposition}\label{prop:frame}
 Let $M$ be a Riemann surface and $U$ a simply-connected open subset of $M$.
Let $\mathbb{F}(z, \bar{z},\lambda): U\rightarrow SU(3)$,  $\lambda\in S^1, z \in U, $ be a solution to the system \eqref{eq:mathbbF}.
Then $[\mathbb{F}(z, \bar{z}, \lambda)e_3]$ gives a minimal Lagrangian surface defined on $U$ with values in $\mathbb{C}P^2$ and 
with the metric $g=2e^{u}dzd\bar{z}$ and the Hopf differential $\Psi^{\nu}=\nu \psi dz^3$.

Conversely, suppose $f^{\nu}:M\rightarrow \mathbb{C}P^2$ is a conformal parametrization of
a minimal Lagrangian surface in $\mathbb{C}P^2$ with the metric $g=2e^{u}dzd\bar{z}$ and Hopf differential
$\Psi^{\nu}=\nu\psi dz^3$. Then for any open, simply-connected subset $U$ of $M$ there exists a unique frame $\mathbb{F}: U \rightarrow SU(3)$ satisfying \eqref{eq:mathbbF} and  $[\mathbb{F}(z, \bar{z},\lambda)e_3] = f$.
\end{proposition}

\begin{remark}
\begin{enumerate}
\item In general, the notion of a \lq\lq frame\rq\rq only denotes maps 
$\F : U \rightarrow SU(3)$ such that  $[\mathbb{F}(z, \bar{z}, \lambda)e_3]$ is a minimal Lagrangian surface. Then two such frames $\F$ and $\hat{\F}$ are in the relation $\hat{\F} = W \F k$ with $W\in SU(3)$ and $k$ a map $k:U \rightarrow U(1)$.

\item  Note that in this paper $U(1)$ acts by diagonal matrices of the form  $\mathrm{diag}(a,a^{-1},1)$ on the right. In particular,  any gauge $k$ for $\F$ is of this form. 
\end{enumerate}

\end{remark}

\subsection{The loop group method for minimal Lagrangian surfaces}

 Let $\sigma$ denote the  automorphism of $SL(3,\mathbb{C})$ of order $6$ defined by
\begin{equation*}
\sigma: g\mapsto P (g^t)^{-1} P^{-1}, \quad
P=\left(
  \begin{array}{ccc}
    0 & \alpha & 0 \\
    \alpha^2 & 0 & 0 \\
    0 &0 & 1\\
  \end{array}
\right), \quad \alpha=e^{2\pi i/3},
\end{equation*}
Let $\tau$ denote the anti-holomorphic involution of $SL(3,\mathbb{C})$ which defines  the real form $SU(3)$, 
$$\tau(g):=(\bar{g}^t)^{-1}.$$
Then the corresponding automorphism $\sigma$ of order $6$ and 
the anti-holomorphic automorphism $\tau$
of $sl(3,\mathbb{C})$ are
\begin{equation*}
\sigma: \xi \mapsto -P\xi^t P^{-1}, \quad \tau: \xi \mapsto -\bar{\xi}^t.
\end{equation*}

By $\mathfrak{g}_l$ we denote the $\epsilon^l$-eigenspace of $\sigma$ in  $\mathfrak{g}^{\mathbb C}$, where $\epsilon=e^{\pi i/3}$.
Explicitly these eigenspaces are given as follows
\begin{equation*}
\begin{split}
\mathfrak{g}_0&=\left\{
                    \begin{pmatrix}
                    a &  &  \\
                     & -a &  \\
                     &  & 0 \\
                \end{pmatrix}
                \mid a\in \C
\right\},
\quad
\mathfrak{g}_1=\left\{\begin{pmatrix}
                    0 & b & 0 \\
                     0 & 0 & a \\
                     a & 0 & 0 \\
                  \end{pmatrix}\mid a,b\in\C
\right\},
\\
\mathfrak{g}_2&=\left\{
                  \begin{pmatrix}
                    0 & 0& a  \\
                     0& 0 & 0 \\
                     0& -a & 0 \\
                  \end{pmatrix} \mid a\in \C
              \right\},
\quad
\mathfrak{g}_3=\left\{
                  \begin{pmatrix}
                    a &  &  \\
                      & a &  \\
                      &  & -2a \\
                  \end{pmatrix} \mid a\in \C
              \right\},\\
\mathfrak{g}_4&=\left\{
                  \begin{pmatrix}
                    0 & 0 & 0  \\
                     0& 0 & a \\
                     -a & 0 & 0 \\
                  \end{pmatrix} \mid a\in \C
            \right\},
\quad
\mathfrak{g}_5=\left\{
                  \begin{pmatrix}
                    0 & 0 & a \\
                     b & 0 & 0 \\
                     0 & a & 0 \\
                  \end{pmatrix} \mid a, b\in \C 
                \right\}.
\end{split}
\end{equation*}
Remark that the automorphism $\sigma$ gives a $6$-symmetric space $SU(3)/U(1)$
and any minimal Lagrangian surface in $\mathbb{C}P^2$ frames a primitive map $\mathbb{F}|_{\lambda=1}: M\rightarrow SU(3)/U(1)$.

Using loop group terminology, we can state (refer to \cite{MM}):
\begin{proposition}
Let $f: \mathbb{D} \rightarrow \mathbb{C}P^2$ be a conformal parametrization 
of a contractible Riemann surface.  
Then the following statements are equivalent:
\begin{enumerate}
\item $f$ is minimal Lagrangian.
\item The moving frame $\mathbb{F}|_{\lambda=1}=(-ie^{-\frac{u}{2}} F_z, -ie^{-\frac{u}{2}}F_{\bar z}, F): \mathbb{D} \rightarrow SU(3)/U(1)$ is primitive.
\item $\mathbb{F}^{-1}d\mathbb{F} =(\lambda^{-1} U_{-1}+U_0)dz+(\lambda V_1+V_0)d\bar{z} 
\subset \Lambda su(3)_{\sigma}$ is a one-parameter family of flat connections.
\end{enumerate}
\end{proposition}

 The general Iwasawa decomposition theorem (\cite{PS}) takes in our case, i.e. for the groups $\Lambda SL(3,\mathbb{C})_{\sigma}$  and $\Lambda SU(3)_{\sigma}$, the following explicit form:

\begin{theorem}[Iwasawa Decomposition theorem of $\Lambda SL(3,\mathbb{C})_{\sigma}$]\label{Thm:Iwasawa}
Multiplication $\Lambda SU(3)_{\sigma}\times \Lambda^{+} SL(3, \C)_{\sigma} \rightarrow \Lambda SL(3, \C)_{\sigma}$ is a diffeomorphism onto. Explicitly, 
every element $g\in \Lambda SL(3,\mathbb{C})_{\sigma}$ can be represented in the form
$g=h V_{+}$ with $h\in \Lambda SU(3)_{\sigma}$ and $V_{+}\in \Lambda^{+} SL(3,\mathbb{C})_{\sigma}$.
One can assume without loss of generality that $V_{+} (\lambda=0)$ has only positive diagonal entries.
In this case the decomposition is unique. 
\end{theorem}

\section{Equivariant minimal Lagrangian surfaces}
\label{Sec:Equiv}

In this section we will investigate minimal Lagrangian immersions for which there exists a one-parameter family  $(\gamma_t, R_t) \in (Aut(M), Iso(\C P^2))$ of symmetries.

\begin{definition}
Let $M$ be any connected Riemann surface and  $f:M \rightarrow \mathbb{C}P^2$ an immersion. Then $f$  is called  equivariant, relative to the one-parameter group $(\gamma_t, R(t)) \in (Aut(M), Iso(\C P^2))$, if 
$$f(\gamma_t \cdot p)=R(t) f(p)$$ for all $p\in M$ and all $t\in \mathbb{R}$.
\end{definition}

By the definition above, any Riemann surface $M$ admitting an equivariant minimal Lagrangian immersion admits a one-parameter group of (biholomorphic) automorphisms.
Fortunately, the classification of such surfaces is very simple:
\begin{theorem} [Classification of Riemann surfaces admitting one-parameter groups of automorphisms, e.g. \cite{Farkas-Kra}]
\mbox{}\par
\begin{enumerate}
\item{$S^2$,}
\item{$\C$, $\mathbf{D}$,}
\item{$\C^*$,}
\item{$\mathbf{D}^*, \mathbf{D}_r$,}
\item{$T=\C/\Lambda_{\tau}$,}
\end{enumerate}
where the superscript \lq\lq \, $^*$\rq\rq denotes deletion of the point $0$, the subscript \lq\lq\, $ r$\rq\rq denotes the open annulus between $0 < r < 1/r$ and 
$\Lambda_{\tau}$ is the free group generated by the two translations $z\mapsto z+1$, $z\mapsto z+\tau$, $\mathrm{Im}\tau >0$.
\end{theorem}

Looking at this classification, one sees that after some composition with some holomorphic transformation one obtains the following picture, including 
the groups of translations:

\begin{theorem}(Classification of Riemann surfaces admitting one-parameter groups of automorphisms and
representatives for the one-parameter  groups, e.g. \cite{Farkas-Kra})
\label{class-equi}

\begin{enumerate}
\item $S^2$, group of rotations about the $z$-axis,
\item
\begin{enumerate}
\item $\C$, group of all real translations,
\item $\C$, group of all rotations about the origin $0$,
\item $\mathbf{D}$, group of all rotations about the origin $0$,
\item $\mathbf{D} \cong \mathbb{H}$, group of all real translations,
\item $ \mathbf{D} \cong \mathbb{H} \cong \log  \mathbb{H} =\St$, the strip between $y=0$ and $y = \pi$, 
group of all real translations,
\end{enumerate}
\item $\C^*$, group of all rotations about $0$,
\item $\mathbf{D}^*, \mathbf{D}_r,$ group of all rotations about $0$,
\item $T$, group of all real translations.
\end{enumerate}

\end{theorem}

For later purposes we state the following 

\begin{definition}
Let $f: M \rightarrow \C P^2$ be an equivariant minimal Lagrangian immersion, then $f$ is called \lq\lq translationally equivariant\rq\rq, if the group of automorphisms acts by (all real) translations. It is called \lq\lq rotationally equivariant\rq\rq, if the group acts by all rotations about $0$.
\end{definition}

\begin{remark} 
Since we know that  any minimal Lagrangian immersion $f$ from a sphere is totally geodesic and it is the standard immersion of $S^2$ into $\mathbb{C}P^2$ (\cite{Yau1974}), we will exclude the case $S^2$ from the discussions in this paper. 
\end{remark}

\section{Translationally equivariant minimal Lagrangian immersions}
\label{Sec:trans equiv}

By what was said just above, we will assume throughout this section that the surface $M$ is a strip $\St$  in $\mathbb{C}$ parallel to the $x$-axis.
Actually, by applying a translation in $y-$direction we can assume that the real axis is 
contained in $\St$ and in particular $0 \in \St$.
 
We thus consider  minimal Lagrangian immersions
$f:\St \rightarrow \mathbb{C} P^2$  for which there exists a one-parameter subgroup $R(t)$ of $SU(3)$ such that 
$$f(t + z, t+\bar{z})=R(t) f(z,\bar{z})$$
for all $z\in \St$. 

Let $\mathbb{F}:\St \rightarrow SU(3)$ be a frame of $f$ satisfying $\mathbb{F}(0)=I$.
Since $f$ is  translationally equivariant we obtain that the frame $\mathbb{F}$ of $f$ is  translationally equivariant in the sense that
\begin{equation}\label{eq:equivframe1}
\mathbb{F}(t + z)=R(t) \mathbb{F}(z, \bar{z}) \mathcal{K}(t,z)
\end{equation}
holds, where $\mathcal{K}(t,z)$
is a crossed homomorphism with values in $U(1)$. 
This means  that 
 $\mathcal{K}$ 
 can be chosen such that $\F(0) = I$ and 
 satisfies the following cocycle condition: 
\begin{equation}\label{eq:cocycle}
\mathcal{K}(t+s,z)=\mathcal{K}(s,z)\mathcal{K}(t,s+z).
\end{equation}
In fact,
\begin{theorem}
$\mathcal{K}(t,z)$ is a coboundary.  More precisely, for the matrix function
$h(z)=\mathcal{K}(x,iy)^{-1}$ we have 
\begin{equation}\label{eq:coboundary}
\mathcal{K}(t,z) = h(z)h(t+z)^{-1}.
\end{equation}
Replacing $h$ by $ \hat{h} = h(0)^{-1} h$ if necessary we can even assume without loss of generality that the coboundary equation above holds with some matrix function $h$ also satisfying $h(0) = I$.
\end{theorem}
\begin{proof}
Setting $z=iy$ in \eqref{eq:cocycle},  we get
\begin{equation}\label{eq:***}
\mathcal{K}(t+s, iy)=\mathcal{K}(s,iy)\mathcal{K}(t,s+iy).
\end{equation}
 Take $h(z)=\mathcal{K}(x,iy)^{-1}$, where $z=x+iy$. Then putting $s=x$ in \eqref{eq:***}, we obtain
\begin{equation*}
h(z)h(t+z)^{-1}=\mathcal{K}(x,iy)^{-1}\mathcal{K}(t+x,iy)=\mathcal{K}(t,z),
\end{equation*}
which completes the proof of \eqref{eq:coboundary}.
The last statement is trivial.
\end{proof}

This implies the important
\begin{theorem}
For any translationally equivariant minimal Lagrangian immersion, the frame $\F$ can be chosen such that
$\F (0) = I$ and 
\begin{equation*}\label{eq:equiv-F_lambda-0}
\mathbb{F}(t+z)=\chi(t)\mathbb{F}(z),
\end{equation*}
holds, where $\chi(t)$ is a one-parameter group in $SU(3)$.
\end{theorem}

\begin{proof}
Choosing $h \in U(1)$ as in the theorem above, satisfying $h(0) = I$ and replacing $\mathbb{F}$ by $\hat{\mathbb{F}}(z):= \mathbb{F}(z)h(z)$,
we obtain  from \eqref{eq:equivframe1} and \eqref{eq:coboundary}
\begin{eqnarray*}
\hat{\mathbb{F}}(t + z)&=&\mathbb{F}(t + z) h(t +z)=R(t) \mathbb{F}(z)\mathcal{K}(t,z)h(t +z)\nonumber\\
&=&R(t) \mathbb{F}(z)h(z) 
=R(t) \hat{\mathbb{F}}(z).
\end{eqnarray*}
Thus $\hat{\F}$ satisfies the claim.
\end{proof}

Let's now consider the frame $\mathbb{F}$ obtained in the theorem above. It satisfies the Maurer-Cartan
equation $\alpha = \mathbb{F}^{-1}d\mathbb{F}=U_{-1}dz+U_{0}dz+V_{0}d\bar{z}+V_{1}d\bar{z}$.
Now we introduce $\lambda\in S^1$ 
as usual. We set 
\begin{equation}\label{eq:F_lambda}
\alpha_\lambda=\lambda^{-1}U_{-1}dz+U_{0}dz+V_{0}d\bar{z}+\lambda V_{1}d\bar{z}.
\end{equation}
Then $\alpha(z,\lambda)$ is integrable, since the Maurer-Cartan form of the original frame (used in the theorem above) is integrable.

Let $\F(z,\lambda)$ denote the solution to

\begin{equation*}
\mathbb{F}(z, {\lambda})^{-1}d\mathbb{F}(z, {\lambda}) = \alpha_\lambda,
\end{equation*}
also satisfying $\F(0,\lambda) = I$.
Then $\F$ satisfies
\begin{equation}\label{eq:equiv-F_lambda}
\mathbb{F}(t+z, \lambda)=\chi(t,\lambda)\mathbb{F}(z,\lambda),
\end{equation}
for any $z\in M$ and a one-parameter group
$\chi(t,\lambda)=e^{tD(\lambda)}$ for some $D(\lambda)\in \Lambda su(3)_{\sigma}$. 
This is the equivariance condition on the extended frame $\mathbb{F}(z,\lambda)$ assumed in  \cite{BuKi}.

But in the context of equivariant minimal Lagrangian immersions it is obvious that the coefficient  matrices of \eqref{eq:mathbbF}
are independent of $x$. 
Therefore, the solution $\F_0$  to the differential equations \eqref{eq:mathbbF} with initial condition $\F_0(0,\lambda) = I$ also satisfies \eqref{eq:equiv-F_lambda}. It is easy to see that two frames satisfying 
\eqref{eq:equiv-F_lambda} only differ by some gauge in $U(1)$ which is independent of $x$. Thus we obtain

\begin{theorem} \label{essential-transfo}
For the extended frame $\F$ of any translationally equivariant minimal Lagrangian immersion we can assume without loss of generality 
$\F(0,\lambda)=I$ and
\begin{equation*}
\mathbb{F}(t+z, \lambda)=\chi(t,\lambda)\mathbb{F}(z,\lambda),
\end{equation*}
with $\chi(t,\lambda)=e^{tD}$ for some $D\in \Lambda su(3)_{\sigma}$. 
Moreover, we can also assume that $\F$  satisfies 
\eqref{eq:mathbbF}. 

Any two frames satisfying \eqref{eq:equiv-F_lambda} only differ by some gauge in $U(1)$  which only depends on y.
\end{theorem}

\subsection{Burstall-Kilian theory for translationally equivariant minimal Lagrangian immersions}
\label{subsect:BK potential}

In this section we assume that the frame is chosen as in 
Theorem \ref{essential-transfo}. Then,
following Burstall-Kilian (\cite{BuKi})
and setting $t=-x$ and $z=x+it$, 
we derive from  \eqref{eq:equiv-F_lambda}, 
\begin{eqnarray}\label{eq:F}
\mathbb{F}(z,\lambda)&=&e^{x D(\lambda)}\mathbb{F}(iy,\lambda). 
\end{eqnarray}
We also assume as before $\F(0,\lambda) = I$.
Then 
\begin{equation*}\label{eq:F-1dF}
\begin{split}
\alpha_{\lambda}=
\mathbb{F}^{-1}(z,\lambda)d\mathbb{F}(z,\lambda)&=\mathbb{F}(iy, \lambda)^{-1} D\mathbb{F}(iy, \lambda)dx+\mathbb{F}(iy,\lambda)^{-1}\frac{d}{dy}\mathbb{F}(iy,\lambda)dy.\\
\end{split}
\end{equation*}
So  
\begin{equation}\label{eq:AB}
A_\lambda(y):=\mathbb{F}(iy, \lambda)^{-1} D(\lambda)\mathbb{F}(iy, \lambda),\quad
B_\lambda(y):=\mathbb{F}(iy, \lambda)^{-1}\frac{d}{dy}\mathbb{F}(iy, \lambda)
\end{equation}
depend only on $y$.
Comparing to \eqref{eq:F_lambda}, we infer
\begin{equation}\label{eq:ABU-1}
\begin{split}
A_{\lambda}(y)&=\lambda^{-1}U_{-1}+U_{0}+V_{0}+\lambda V_{1},\\
B_{\lambda}(y)&=i(\lambda^{-1}U_{-1}+U_{0}-V_{0}-\lambda V_{1}),
\end{split}
\end{equation}
where we have used 
$$\alpha_{\lambda}=\lambda^{-1} U_{-1}dz+U_0 dz+V_0d\bar{z}+\lambda V_1 d\bar{z}.$$
Hence $A_{\lambda}(y), B_{\lambda}(y), D(\lambda)\subset \Lambda_1 su(3)_{\sigma}$ and
$\alpha_{\lambda}(\partial/\partial \bar{z})$ is holomorphic in $\lambda$. 
Set
$b_{\lambda}(y):=e^{-iyD(\lambda)} \mathbb{F} (iy, \lambda)$. Then its Maurer-Cartan form is given by $b_{\lambda}^{-1} db_{\lambda}=-2i \alpha_{\lambda}(\frac{\partial}{\partial \bar{z}})dy$.

Therefore, with the initial condition $\F(0,\lambda) = I$, we know that $b_{\lambda}$ is holomorphic in $\lambda$ and $b_{\lambda}\in \Lambda^{+}SL(3, \mathbb{C})_{\sigma}$.
It follows that
$$\mathbb{F}(z, {\lambda})=e^{xD(\lambda)}\mathbb{F}(iy, {\lambda})=e^{zD(\lambda)}b_{\lambda}(y)$$ 
is an Iwasawa decomposition of $\mathbb{F}$.
This means that $\mathbb{F}$ is generated by the degree one constant potential $D(\lambda)=A_{\lambda}(0) \in \Lambda su(3)_{\sigma}$.

Conversely,
for any constant degree one potential $D(\lambda)\in \Lambda su(3)_{\sigma}$, we have the solution 
$C(z,\lambda):=e^{zD(\lambda)}=e^{xD(\lambda)} e^{iyD(\lambda)}$ to $dC = CD(\lambda)dz$, $C(0,\lambda)=I$.
Assume that an Iwasawa decomposition of $e^{iyD(\lambda)}$ is given by 
\begin{equation}\label{Iwasawa-e^iyD}
e^{iyD(\lambda)}=U(y,\lambda)U_{+}(y,\lambda),
\end{equation}
where $U(y,\lambda): M \rightarrow \Lambda SU(3)_{\sigma}$
and $U_{+}(y,\lambda)\in \Lambda^{+} SL(3,\mathbb{C})_{\sigma}$.
Because $e^{xD(\lambda)}\in \Lambda SU(3)_{\sigma}$ for all $x\in \mathbb{R}$, we conclude that
an Iwasawa decomposition of $e^{zD(\lambda)}$ is given by
$e^{zD(\lambda)}=\mathbb{F}(z,\lambda)U_{+}(y,\lambda)$, where
\begin{equation}\label{eq:FU}
\mathbb{F}(z,\lambda)=e^{xD(\lambda)}U(y,\lambda).
\end{equation}
Hence, $\mathbb{F}(z,\lambda)$ is translationally equivariant.

Thus we conclude 
\begin{proposition} 
A minimal Lagrangian surface in $\mathbb{C}P^2$ is translationally equivariant if and only it is generated by a degree one constant potential
$D(\lambda)dz$.
In this case the immersion can be defined without loss of generality on all of $\C$. 
The potential function $D(\lambda)$ can be obtained from the extended frame $\F$ satisfying \eqref{eq:equiv-F_lambda} and $\F(0,\lambda)=I$ by the equation
$$D(\lambda)=\F(z,\lambda)^{-1}\partial_x\F(z,\lambda)|_{z=0}.$$
\end{proposition}

\begin{remark}
Since any two frames satisfying   \eqref{eq:equiv-F_lambda}  and attaining $I$ at $z =0$ also satisfy equation \eqref{eq:F}, it is easy to see that $D(\lambda)$ is uniquely determined. From this it also follows again that two such frames only differ by some gauge $k(iy) \in U(1)$.

\end{remark}

\section{Explicit Iwasawa decomposition for translationally equivariant minimal Lagrangian immersions}
\label{Sec:Iwasawa}

\subsection{The basic set-up}

We have seen above in Section \ref{subsect:BK potential} that every translationally equivariant minimal Lagrangian immersion can be obtained from some potential of the form 
\begin{equation*}
\eta = D(\lambda) dz,
\end{equation*}
where
\begin{equation}\label{eq:D}
D(\lambda) = \lambda^{-1} D_{-1} + D_0 + \lambda D_1 \in \Lambda su(3)_\sigma.
\end{equation}

The general loop group approach requires to consider the solution to $dC = C \eta, C(0,\lambda ) = I$. This is easily achieved by
$C(z,\lambda) = \exp(zD(\lambda)).$

Next one needs to perform an Iwasawa splitting.
In general this is very complicated and difficult to carry out explicitly.
But, for translationally equivariant minimal Lagrangian surfaces in $\mathbb{C}P^2$,  one is able to carry out an explicit Iwasawa decomposition of $\exp(zD(\lambda))$. 

From \eqref{eq:F}, \eqref{Iwasawa-e^iyD} and \eqref{eq:FU}, we get
\begin{equation}\label{eq:F(iy)}
F(iy,\lambda)=U(y,\lambda)=e^{iyD(\lambda)} U_{+}(y,\lambda)^{-1}.
\end{equation}
Substituting \eqref{eq:F(iy)} into \eqref{eq:AB}, we obtain
\begin{equation}\label{eq:A_lambda}
\begin{split}
A_{\lambda}(y)&=U_{+}(y,\lambda)D(\lambda)U_{+}(y,\lambda)^{-1},\\
B_{\lambda}(y)&=U_{+}(y,\lambda)i D(\lambda)U_{+}(y,\lambda)^{-1}- 
\frac{d}{dy}U_{+}(y,\lambda)   U_{+}(y,\lambda)^{-1}.
\end{split}
\end{equation}

Comparing this to \eqref{eq:ABU-1}, we obtain the equations
\begin{equation} \label{eq:Equ1}
U_{+}(y,\lambda)D(\lambda)U_{+}^{-1}(y,\lambda)=
\lambda^{-1}U_{-1}+U_0+\lambda V_1+V_0
=:\Omega,
\end{equation}
\begin{equation} \label{eq:Equ2}
\frac{d}{dy}U_{+}(y,\lambda) U_{+}(y,\lambda)^{-1}=
2i(\lambda V_1+V_0).
\end{equation}

It is important to note that
because $U_{+}$ only depends on $y$ and $\F$ satisfies \eqref{eq:mathbbF}, the matrix $\Omega$  
is of the form
\begin{equation*}\label{eq:Omega}
\Omega
=\begin{pmatrix}
\frac{u_z-u_{\bar z}}{2}&-i\lambda \bar{\psi}e^{-u}&i\lambda^{-1}e^{\frac{u}{2}}\\
-i\lambda^{-1}\psi e^{-u}&-\frac{u_z-u_{\bar z}}{2}&i\lambda e^{\frac{u}{2}} \\
i\lambda e^{\frac{u}{2}}&i\lambda^{-1} e^{\frac{u}{2}}&0
\end{pmatrix},
\end{equation*}
where $u$ only depends on $y$ and $\psi$ is constant.

The equations \eqref{eq:Equ1} and \eqref{eq:Equ2}
are the basis for an explicit computation of the Iwasawa decomposition of 
$\exp(z D(\lambda))$. There will be two steps:

{\bf Step 1:} Solve  equation \eqref{eq:Equ1} 
by some matrix $Q$.
Then $U_+$ and $Q$ satisfy 
\begin{equation*}
U_+ = Q E, \hspace{2mm} \text{ where}  \hspace{2mm} E \hspace{2mm}  \text{commutes with} \hspace{2mm} D.
\end{equation*}

{\bf Step 2:} Solve  equation \eqref{eq:Equ2}. This will generally only mean to carry out two integrations in one variable. 

\subsection{Evaluation of the characteristic polynomial equations}

Step 1 mentioned above actually consists of two sub-steps.
First of all one determines $\Omega$ from $D$ and then one computes a solution $Q$ the equation \eqref{eq:Equ1}.

In this section we will discuss the first sub-step. 
In our case we observe that $D$ and $\Omega$ are conjugate and therefore have the same characteristic polynomials.
Using the explicit form of $\Omega$ stated just above and 
writing $D$ in the form 
\begin{eqnarray*}
D=\begin{pmatrix}
\alpha &-\lambda \bar{b} & \lambda^{-1} a \\
\lambda^{-1} b &-\alpha&-\lambda \bar{a}\\
-\lambda \bar{a} &\lambda^{-1} a &0
\end{pmatrix}\in \Lambda_1 \subset \Lambda su(3)_{\sigma},
\end{eqnarray*}
where $\alpha, a$ and $b$ are constants,
\eqref{eq:Equ1} leads to
\begin{eqnarray}
&&2e^{u}+|\psi|^2e^{-2u}+\frac{1}{4}(u^\prime)^2=-\alpha^2+2|a|^2+|b|^2=:\beta,\label{eq:equiv1}\\
&&\psi=-ia^2b,\label{eq:equiv2}
\end{eqnarray}
where $\alpha, a, b$ and $\psi$ are constants. 

\begin{remark} We have seen that if $\psi \equiv 0,$ then the surface is totally geodesic, hence the image of $f$ is an open portion of the real projective plane.
We will ignore this case from now on and will assume that $\psi\not\equiv 0.$
\end{remark}

\subsection{Explicit solutions for metric and cubic form}
\label{subsec:metric}

If $u^{\prime}\equiv 0$, then $u$ is constant and the surface is flat. 
It is well known that the only flat minimal Lagrangian surface in $\C P^2$ is an open subset of the Clifford torus up to isometries of $\C P^2$ (\cite{LOY}). In the following we will assume $u^{\prime}\not \equiv 0$.

Notice that \eqref{eq:equiv1} is a first integral of the Gauss equation
\begin{equation}\label{eq:GaussODE0}
\frac{1}{4}u^{\prime\prime}+e^u-|\psi|^2 e^{-2u}=0.
\end{equation}

Making the change of variables $w=e^{u}$ in \eqref{eq:equiv1}, we obtain equivalently
\begin{equation}\label{eq:w0}
(w^{\prime})^2+8w^3-4 \beta w^2 +4|\psi|^2=0.
\end{equation}

Since we assume $\psi \neq 0$, the solutions to \eqref{eq:w0} are given in terms of bounded Jacobi elliptic functions.
Since all Jacobi elliptic functions are periodic, there exists a point,  where the derivative of $u$ vanishes. 
Choosing this point as the origin, we can always assume $u^{\prime}(0)=0$.
For our loop group setting this has an important consequence:

\begin{theorem}
By choosing the coordinates such that the metric for a given translationally equivariant minimal Lagrangian immersion has a vanishing derivative at $z = 0$, we obtain that the generating matrix $D$ in \eqref{eq:D} satisfies $D_0 = 0$.
\end{theorem}

We will therefore always assume this condition.
This convention in combination with \eqref{eq:equiv1}  implies
\begin{equation}\label{eq:equia1beta}
2a_1+\frac{|\psi|^2}{a_1^2}=\beta,
\end{equation}
where $a_1:=e^{u(0)}>0$.

The following computations are very similar to the ones given in \cite{CU94}. We include them for the convenience of the reader.

Using \eqref{eq:equia1beta} it is easy to verify that  \eqref{eq:w0} can be rewritten in the form
\begin{equation}\label{eq:w'}
(w^\prime)^2+8(w-a_1)(w-a_2)(w+a_3)=0,
\end{equation}
where
\begin{eqnarray*}
a_1&=&e^{u(0)}>0,\\
a_2&=& \frac{\frac{\beta}{2}-a_1+\sqrt{(\frac{\beta}{2}-a_1)^2+4a_1 (\frac{\beta}{2}-a_1)}}{2}>0,\\
a_3&=& \frac{-(\frac{\beta}{2}-a_1)+\sqrt{(\frac{\beta}{2}-a_1)^2+4a_1(\frac{\beta}{2}-a_1)} }{2}>0.
\end{eqnarray*}
Since $\frac{\beta}{2}-a_1=\frac{|\psi|^2}{2a_1^2}>0$, we know $a_2> a_3$. 
Notice that $a_1=a_2$ if and only if  $\beta=3|\psi|^{2/3}$, then  \eqref{eq:GaussODE0} has the unique solution $u \equiv \frac{2}{3} \log |\psi|$, which conflicts our starting assumption $u^{\prime}\not\equiv 0$. Therefore we can assume without loss of generality that $a_1>a_2$ holds.

Then with $a_2<w<a_1$, \eqref{eq:w'} leads to
$$dy=\frac{1}{\sqrt{-8(w-a_1)(w-a_2)(w+a_3)}}dw.$$
Integrating gives 
\begin{equation*}
y=-\frac{1}{\sqrt{2(a_1+a_3)}}J(\arcsin \sqrt{\frac{a_1-w}{a_1-a_2}}, k),
\end{equation*}
where 
\begin{equation}\label{eq:k}
k^2=\frac{a_1-a_2}{a_1+a_3}
\end{equation}
and $J$ denotes the elliptic integral of the first kind
$$J(\theta,k)=\int_0^{\theta}\frac{d\alpha}{\sqrt{1-k^2\sin^2\alpha}}, $$
for any $0\leq k\leq 1$.
Thus the solution to \eqref{eq:w0} is 
\begin{equation}\label{eq:equiv_u0}
w(y)=e^{u(y)}=a_1(1-q^2\mathrm{sn}^2(ry,k)),
\end{equation}
where 
\begin{equation}\label{eq:q&r}
q^2=\frac{a_1-a_2}{a_1}, \quad r=\sqrt{2(a_1+a_3)}.
\end{equation}
It is easy to see that the solution $u(y)$ is an even periodic function with period $2T$, where $T=\frac{K}{r}$ and $K=J(\frac{\pi}{2}, k)$ is the complete elliptic integral of the first kind. 
Thus for any (in $x-$direction) translationally equivariant minimal Lagrangian surface in $\mathbb{C}P^2$, its metric conformal factor $e^u$ is given by  \eqref{eq:equiv_u0} in terms of a Jacobi elliptic function and its cubic Hopf differential is constant and given by \eqref{eq:equiv2}.


\subsection{Explicit Iwasawa decompositions} 

Recall that for translationally equivariant minimal Lagrangian surfaces, the potential matrix $D(\lambda)$ coincides with $A_{\lambda}(0)=\Omega|_{y=0}$ in \eqref{eq:A_lambda}, so we have 
(including the convention above about the origin)
\begin{equation*}\label{eq:D=Alambda0}
D(\lambda)=\begin{pmatrix}
0&-i\lambda\bar{\psi}e^{-u(0)}&i\lambda^{-1}e^{\frac{u(0)}{2}}\\
-i\lambda^{-1}\psi e^{-u(0)}&0&i\lambda e^{\frac{u(0)}{2}}\\
i\lambda e^{\frac{u(0)}{2}}&i\lambda^{-1} e^{\frac{u(0)}{2}}&0
\end{pmatrix},
\end{equation*}
where  $\alpha=-\frac{iu^{\prime}(0)}{2}=0$, $a=ie^{\frac{u(0)}{2}}$
and $b=-i\psi e^{-u(0)}$.
We may summarize the following proposition:
\begin{proposition} Up to isometries in $\mathbb{C}P^2$, any translationally equivariant minimal Lagrangian surface can be generated by a potential of the form 
\begin{equation}\label{eq:Dalpha=0}
\begin{pmatrix}
0&-\lambda \bar{b}&\lambda^{-1}a\\
\lambda^{-1}b&0&-\lambda \bar{a}\\
-\lambda\bar{a}&\lambda^{-1} a&0
\end{pmatrix} dz,
\end{equation}
where $a$ is purely imaginary and $b=\frac{i\psi}{a^2}$ are constants.
\end{proposition}

Thus the characteristic polynomial of $D(\lambda)$ in \eqref{eq:Dalpha=0} is given by
\begin{equation} \label{eq:charpol}
\det(\mu I-D(\lambda))=\mu^3+\beta \mu-2i\mathrm{Re}(\lambda^{-3}\psi),
\end{equation}
where 
\begin{equation} \label{D-eig}
\begin{split}
\beta&=2|a|^2+|b|^2=2e^{u(0)}+|\psi|^2 e^{-2u(0)}:=2a_1+\frac{|\psi|^2}{a_1^2}.
\end{split}
\end{equation}

\begin{lemma} \label{D-roots}
 With the notation introduced above we have
\begin{enumerate}
\item  The characteristic polynomial \eqref{eq:charpol}  of $D(\lambda)$ has purely imaginary roots which depend on $\lambda$, 
but not on $z$.
\item For any non-flat minimal Lagrangian surface in $\C P^2$, the characteristic polynomial \eqref{eq:charpol} of $D(\lambda)$ 
has three distinct roots for any choice of $\lambda \in S^1$.
 Moreover, the root $0$ occurs if and only if  $\lambda^{-3}\psi$ is purely imaginary. This case can only happen for six different values of $\lambda$.
\end{enumerate}
\end{lemma}

\begin{proof}
Claim 1 simply follows from the observation that  entries of $D(\lambda)\in su(3)$ only depend on $\lambda\in S^1$. 
The characteristic polynomial \eqref{eq:charpol} of $D(\lambda)$ has three distinct purely imaginary roots if and only if its discriminant satisfies
$$(\frac{\beta}{3})^3-[\mathrm{Re}(\lambda^{-3}\psi)]^2>0.$$

Regarding $\beta$ as a function of $a_1\in (0,\infty)$ it is easy to see that $\beta$ attains the minimum value $3|\psi|^{\frac{2}{3}}$ when $a_1=|\psi|^{\frac{2}{3}}$.  In this case $D(\lambda)$ has multiple eigenvalues and $a_2=a_1=|\psi|^{\frac{2}{3}}$, $a_3=|\psi|^{\frac{2}{3}}/2$. It follows from \eqref{eq:w'} that the corresponding surface is flat.  We have excluded this case. 
Therefore we have $\beta>3|\psi|^{\frac{2}{3}}$, which completes the proof of Claim 2.
\end{proof}

Now take 
\begin{equation}\label{eq:Q_0}
Q_0=\mathrm{diag}(i a^{-1} e^{\frac{u}{2}}, -i a e^{-\frac{u}{2}},1),
\end{equation}
such that 
$$\hat{\Omega}=Q_0^{-1}\Omega Q_0=\begin{pmatrix}
-\frac{iu^{\prime}}{2}&i\lambda \bar{\psi}a^2 e^{-2u}&\lambda^{-1}a\\
\lambda^{-1}b&\frac{iu^{\prime}}{2}&-\lambda a^{-1}e^u\\
-\lambda a^{-1}e^u&\lambda^{-1}a&0
\end{pmatrix}$$
has the same coefficients at $\lambda^{-1}$ as $D(\lambda)$.

Then by a straightforward  computation, we solve $\tilde{Q}D(\lambda)\tilde{Q}^{-1}= \hat{\Omega}$ by the following matrix
\begin{equation}\label{eq:checkQ}
\tilde{Q}=\frac{\lambda^3}{\kappa}\begin{pmatrix}
\check{p}&\check{q}&\check{v}_1\\ 
\check{s}&\check{t}&\check{v}_2\\
0&0&\check{c}
\end{pmatrix},
\end{equation}
where
\begin{equation}\label{eq:pqstvc}
\begin{split}
\check{p}&= -|a|^2 \frac{u^{\prime}}{2}+\lambda^3 \bar{\psi}|a|^2 e^{-u}-\lambda^{-3}\psi,\\
\check{q}&= \frac{\lambda^{-2}a}{\bar{a}}[\frac{u^{\prime}}{2} |a|^2 -\lambda^3\bar{\psi}e^{-u}(|a|^2-e^u)],\\
\check{s}&=\frac{\lambda^2}{a^2}[|a|^2 \frac{u^{\prime}}{2}e^u+\lambda^{-3} \psi (|a|^2- e^u)],\\
\check{t}&=\frac{1}{|a|^2} (-|a|^2 \frac{u^{\prime}}{2}e^u+\lambda^3 \bar{\psi}e^u -\lambda^{-3}\psi |a|^2),\\
\check{v}_1&=-2i\lambda^{-1}a (|a|^2-e^u),\\
\check{v}_2&=-2i\lambda a^{-1} e^u (|a|^2-e^u),\\
\check{c}&=\lambda^3 \bar{\psi}-\lambda^{-3}\psi-e^u u^{\prime},\\
\kappa&=(\lambda^6 \bar{\psi}-\psi -\lambda^3e^u u^{\prime})^{2/3}( \lambda^6 \bar{\psi}-\psi)^{1/3}.
\end{split}
\end{equation}

Moreover, $\det \tilde{Q}=1$ and  $\tilde{Q}(0,\lambda)=Q_0(0,\lambda)=I$ due to $a=ie^{\frac{u(0)}{2}}$.
If $\lambda$ is small, the denominator of the coefficient of $\tilde{Q}$ is single-valued.
Altogether we have found a solution to 
equation \eqref{eq:Equ1} by $Q=Q_0 \tilde{Q}$.

\bigskip

Since also $U_+$ has the same properties, we obtain that $E = Q^{-1} U_+$ 
has determinant $1$, attains the value $I$ for $z=0$, is holomorphic for all small $\lambda$ and satisfies
$[Q^{-1}U_{+}, D]=0.$

By Lemma \ref{D-roots}  
 we can assume without loss of generality that 
 $D=D(\lambda )$  is regular semi-simple for all but finitely many values of $\lambda$. Therefore, for all small $z$ and small $\lambda$ we can write
$E = \exp ( \mathcal{E} ) $, where $[\mathcal{E}, D]=0.$

Since, in the computation of $Q$, we did not worry about the twisting condition, the matrix $\mathcal{E}$ is possibly an untwisted loop matrix in $SL(3,\C).$ But since $SL(3,\C)$ has rank $2$, for any regular semi-simple 
matrix $D=D(\lambda),$ the commutant of $D(\lambda)$ is spanned by $D(\lambda)$ and one other matrix.

\begin{lemma} Every element in the commutant 
$\{X\in \Lambda sl(3,\mathbb{C})_{\sigma}: [X,D]=0\}$
of $D(\lambda)$ has the form 
$X(\lambda)=\kappa_1(\lambda) D(\lambda)+\kappa_2(\lambda)L_0(\lambda)$ with 
$\kappa_1(\epsilon\lambda)=\kappa_1(\lambda)$, $\kappa_2(\epsilon\lambda)=-\kappa_2 (\lambda)$,
where $L_0=D^2(\lambda) -\frac{1}{3} \mathrm{tr}(D^2)  I$.
\end{lemma}

\begin{corollary}
The matrix $Q^{-1} U_+$ has the form
\begin{equation*}
Q^{-1}U_{+}=\exp(\beta_1 D+\beta_2 L_0),
\end{equation*}
where $\beta_1$ and $\beta_2$ are functions of $y$ and $\lambda$ near $0$.
\end{corollary}

With this description of $U_+$  equation \eqref{eq:Equ2} leads to the following two equations:
\begin{eqnarray*}
&-\beta_1^{\prime}\lambda \bar{a} +\beta_2^{\prime}\lambda^{-2}ab=-\frac{2i\lambda e^u}{a}\frac{\check{p}}{\check{c}},\label{eq:31}\\
&\beta_1^{\prime} \lambda^{-1}a+\beta_2^{\prime}\lambda^2 \bar{a}\bar{b}=-\frac{2i\lambda e^u}{a}\frac{\check{q}}{\check{c}}.\label{eq:32}
\end{eqnarray*}

Integrating then yields
\begin{equation}\label{eq:equiv_beta}
\begin{split}
\beta_1(y)&= \int_0^y \frac{2i\lambda^3\bar{\psi}- i u^{\prime} e^u}{\lambda^3\bar{\psi}-\lambda^{-3}\psi-e^u u^{\prime}}ds, \\
\beta_2(y) &= \int_0^y \frac{2e^u}{\lambda^3\bar{\psi}-\lambda^{-3}\psi-e^u u^{\prime}}ds. 
\end{split}
\end{equation}

Putting everything together we obtain

\begin{theorem} [Explicit Iwasawa decomposition]
The extended frame for the translationally equivariant minimal Lagrangian surface in $\mathbb{C}P^2$ generated by the  potential $D(\lambda)dz$  with vanishing diagonal satisfying $ab \neq 0$ is given by
\begin{equation}\label{eq:extended frame F}
\mathbb{F}(z,\lambda)=\exp(zD-\beta_1(y,\lambda) D-\beta_2(y,\lambda) L_0) Q^{-1}(y,\lambda),
\end{equation}
with $\beta_1, \beta_2$ as in \eqref{eq:equiv_beta} and $Q=Q_0\tilde{Q}$ as in \eqref{eq:Q_0}, \eqref{eq:checkQ}, \eqref{eq:pqstvc} and $u$ as in  \eqref{eq:equiv_u0}.
\end{theorem}

\begin{remark}
In the proof of the last theorem we have derived the equation
 $U_{+}=Q\exp(\beta_1 D+\beta_2 L_0).$
In this equation each separate term is only defined for small $\lambda$ and a restricted set of $y'$s. However, due to the globality and the uniqueness of the Iwasawa splitting, the matrix $U_+$ is defined for all $\lambda$ in $\C^*$ and all $z \in \C$.
\end{remark}

\subsection{Explicit expressions for minimal Lagrangian immersions}

To make formula \eqref{eq:extended frame F} explicit we need to know how 
the exponential factor acts on $Q^{-1}e_3$. This can be done in two ways: 
Since the exponential factor commutes with $D$, one can express it in terms of a linear combination of the matrices $I,D,D^2$. Once the coefficients are known, the horizontal lift $F$ is given explicitly.
The second way is to diagonalize $D$ and to expand $Q^{-1}e_3$ relative to
an eigenvector basis of $D$. It turns out that this second approach can be carried out quite easily and yields a straightforward comparison with the work of Castro-Urbano (\cite{CU94}) which we will discuss in the next section.
We would like to point out that in these computations we ignore any \lq\lq twisting\rq\rq.

We start by computing an eigenvector basis for $D$. Let $\mu$ be an eigenvalue of $D$. Then by \eqref{eq:charpol} and \eqref{D-eig} it is easy to verify that the vector 

\begin{equation*} \label{eigenvector for $D$}
s_\mu = 
\begin{pmatrix}
|a|^2 + \mu^2\\
\lambda^{-1}b \mu + \lambda^2 \bar{a}^2\\ 
\lambda^{-2} ab - \lambda \bar{a} \mu
\end{pmatrix}
\end{equation*}
is an eigenvector for $D$ for the eigenvalue $\mu$.
We know from Lemma \ref{D-roots} that for any non-flat minimal Lagrangian surface, up to possibly six values of $\lambda$ the matrix $D$ has three different nonzero eigenvalues. Since $D$ is skew-Hermitian, we also know that the corresponding eigenvectors are automatically perpendicular. Therefore there exists a unitary matrix $L$ such that 
$D=L\mathrm{diag}(\mu_1,\mu_2,\mu_3) L^{-1}$, where, as before, $\mu_j \, (j=1,2,3)$ denote eigenvalues of $D$. As a consequence, for the extended horizontal lift $F$ we thus obtain
$$ F = \F e_3=L\exp(z\Lambda-\beta_1\Lambda-\beta_2(\Lambda^2-\frac{\mathrm{tr}\Lambda^2}{3}I))L^{-1}Q^{-1}e_3,$$
where $\Lambda=\mathrm{diag}(\mu_1,\mu_2,\mu_3)$. 

From \eqref{eq:Q_0}, \eqref{eq:checkQ} and \eqref{eq:pqstvc}, it is easy to derive

$$Q^{-1}e_3=
\frac{1}{\kappa}
\begin{pmatrix}2i\lambda^{-1}a(|a|^2-e^u)\\2i \lambda\bar{a}(|a|^2-e^u)\\ \lambda^3\bar{\psi}-\lambda^{-3}\psi\end{pmatrix},$$
where $ \kappa = (\lambda^3\bar{\psi}-\lambda^{-3}\psi-e^u u^{\prime})^{1/3}(\lambda^3\bar{\psi}-\lambda^{-3}\psi)^{2/3}$.
Since we will eventually project to $\C P^2$, the factor $\kappa$ is actually irrelevant.  

Setting 
$l_j = \frac{s_j}{||s_j||}$, where we put $s_\mu = s_j$ if $ \mu = \mu_j$, we obtain
$$ L= ( l_1,l_2,l_3) \hspace{2mm} \mbox {and} \hspace{2mm} L^{-1} = \bar{L}^t. $$ 

Altogether we have shown

\begin{theorem} Every translationally equivariant minimal Lagrangian immersion generated by the potential $D(\lambda) dz$ has a canonical horizontal lift $F = F(z,\lambda )$ of the form
\begin{equation*}
F(z,\lambda)= 
e^{g_1( z,\lambda)} \langle l_1, Q^{-1}e_3\rangle l_1 +
e^{g_2(z,\lambda)} \langle l_2, Q^{-1}e_3\rangle l_2 + e^{g_3(z,\lambda)} \langle l_3, Q^{-1}e_3 \rangle l_3,
\end{equation*}
where 
\begin{equation*}
g_j (z,\lambda) = z\mu_j(\lambda) - \beta_1(y, \lambda) \mu_j(\lambda) - \beta_2(y,\lambda)(\mu_j(\lambda)^2 - \frac{1}{3}( \mu_1^2 + \mu_2^2 + \mu_3^2) ).
\end{equation*}
\end{theorem}

\section{Equivariant cylinders and tori}
\label{Sec:equiv cylinder tori}

\subsection{Translationally equivariant minimal Lagrangian cylinders}

Based on the description of the frames of (real) translationally equivariant minimal Lagrangian surfaces,
in this  section we will investigate for which (generally complex) periods such an immersion is periodic.
\begin{definition}
Let $f:\mathbb{D} \rightarrow \C P^2$ be a (relative to translations by real numbers)  translationally equivariant minimal Lagrangian surface. 
Then $f$ is called an \emph{equivariant cylinder}, if there exists some complex number $\omega$ such that $f(z + \omega) = f(z)$ for all $z \in \mathbb{D} $.
In this case, $\omega$ is called a \emph{period of $f$}.
If $f$ satisfies this equation for two (over $\R$)  linearly independent 
periods, then $f$ will be called an \emph{equivariant torus}.
\end{definition}

Clearly, every period $\omega$ of some translationally equivariant minimal Lagrangian immersion 
leaves the metric invariant. Since the metric is periodic with (smallest) period $2T$, it follows that the imaginary part of 
$\omega$ is an integer multiple of  $2T$.
Hence we will only consider translations of the form
$$z\mapsto z+p+m2Ti, \quad \text{ with } p \in \mathbb{R}, m\in \mathbb{Z}.$$

From \eqref{eq:pqstvc}  we derive by inspection that $Q$ is invariant under the above translation by $p+m2Ti$.
Therefore, in view of formula  \eqref{eq:extended frame F} for the extended frame we obtain that the monodromy matrix is determined completely by its exponential factor. 

From the properties of $u$ we derive the following properties of $\beta_1$
and $\beta_2$:
\begin{lemma}\label{lem:beta}
\begin{enumerate}
\item  $\beta_j(y+m2T,\lambda)=\beta_j(y,\lambda)+m\beta_j(2T,\lambda)$ for $m\in \mathbb{Z}$ and $j=1,2$.

\item $\beta_1(2T,\lambda)-\overline{\beta_1(2T,\lambda)}=4iT$, $\beta_2(2T,\lambda)+\overline{\beta_2(2T,\lambda)}=0$.

\item $\mathrm{Re}\beta_1(2T,\epsilon\lambda)=\mathrm{Re}\beta_1(2T,\lambda)$,
$\mathrm{Im}\beta_2(2T,\epsilon\lambda)=-\mathrm{Im}\beta_2(2T,\lambda)$,
where $\epsilon=e^{\pi i/3}$ is a sixth root of unity as in the definition of the twisted loop group.

\end{enumerate}
\end{lemma}

As a consequence,  the monodromy matrix 
of the extended frame $\F(z, \lambda)$ for the translation by $\omega = p +  m2Ti $ is
\begin{equation*}
\mathbb{F}(z+p+m2Ti,\lambda)=M(\lambda)\F(z, \lambda),
\end{equation*}
where
\begin{equation*}
M(\lambda)=\exp(pD(\lambda)-m \mathrm{Re}\beta_1(2T,\lambda) D(\lambda)-im\mathrm{Im}\beta_2(2T,\lambda)L_0(\lambda)).
\end{equation*}
Moreover, $M(\lambda)\in\Lambda SU(3)_\sigma $ for any $\lambda \in S^1$.

Thus every translation $\omega=p+m2Ti$, $p\in \R$, $m\in \Z$, induces a symmetry of the translationally equivariant minimal Lagrangian surface constructed from $D(\lambda)$.

Let $id_1(\lambda), id_2(\lambda), id_3(\lambda)$ denote the eigenvalues of $D$. Recalling $\beta$ from \eqref{D-eig}, we see that the monodromy $M(\lambda)$ of the translation $\omega = p + m2Ti$ has the eigenvalues
\begin{eqnarray} \label{eig-mono}
i \lbrace pd_j(\lambda) -m[\mathrm{Re}\beta_1(2T,\lambda) d_j(\lambda)+\mathrm{Im}\beta_2(2T,\lambda)(-d_j(\lambda)^2+\frac{2\beta}{3})] \rbrace
\end{eqnarray}
for $j=1,2,3$.

As a consequence it is easy to obtain
\begin{theorem} \label{equicyl}
For $\lambda = \lambda_0$ the following statements are equivalent.

\begin{enumerate}
\item  The minimal Lagrangian immersion $f(z, \lambda_0)$ is
an equivariant minimal Lagrangian cylinder relative to translation by $\omega =p + m2Ti.$

\item  The monodromy matrix $M(\lambda)$ of the translation by $\omega=p + m2Ti$ satisfies for $\lambda = \lambda_0$ the equation
$M(\lambda_0) = I$.

\item  For the eigenvalues of the monodromy matrix $M(\lambda)$
of the translation by $\omega=p + m2Ti$, the 
following relation holds for 
$j = 1,2$  and $\lambda = \lambda_0$ and integers $l_1, l_2$
\begin{eqnarray} \label{mLi-cyl3}
pd_j(\lambda_0)-m[\mathrm{Re}\beta_1(2T,\lambda_0) d_j(\lambda_0)+\mathrm{Im}\beta_2(2T, \lambda_0) (-d_j(\lambda)^2+\frac{2\beta}{3})]=2l_j \pi.
\end{eqnarray}

\item  In addition we note:
 If $d_1\neq d_2$ and $\lambda = \lambda_0,$ the following relations,
for appropriate  integers $ l_1$ and $l_2$, are equivalent with the relations above
\begin{equation} \label{periodicity-cyl4}
\begin{split}
&m \mathrm{Im}\beta_2(2T, \lambda_0)[d_1(\lambda_0) - d_2(\lambda_0)]
\lbrace d_1(\lambda_0)d_2(\lambda_0)+\frac{2\beta}{3}\rbrace
=2\pi (l_1 d_2(\lambda_0) -l_2 d_1(\lambda_0)),\\
&(d_1-d_2)\{p-m  \mathrm{Re}\beta_1(2T, \lambda_0) +m \mathrm{Im}\beta_2(2T, \lambda_0)(d_1(\lambda_0)+d_2(\lambda_0))\}
=2(l_1-l_2)\pi.
\end{split}
\end{equation}

\end{enumerate}

\end{theorem}

There are two particularly simple choices of translations $\omega =p + m2Ti$, namely purely real and purely imaginary translations. Consequently we obtain:

\begin{corollary} \label{specialperiods}
Retaining the assumptions and the notation of  
Theorem \ref{equicyl}  for the translation $\omega =p + m2Ti$ and the fixed value $\lambda = \lambda_0$, we obtain two natural cases:
\begin{enumerate}

\item{Real translations: If $m = 0$, then $f(z,\lambda_0)$ is an equivariant cylinder if and only if $d_1(\lambda_0)/d_2(\lambda_0)$ is rational.}

\item{Purely imaginary translations: If $p = 0$, then $f(z,\lambda_0)$ is an equivariant cylinder if and only if 
\begin{equation*} \label{imag-transl}
\mathrm{Re}\beta_1(2T,\lambda_0) d_j(\lambda_0)+\mathrm{Im}\beta_2(2T,\lambda_0)(-d_j(\lambda_0)^2+\frac{2\beta}{3}) = 2 \pi r_j,
\end{equation*}
where $r_j$ $(j = 1,2)$ are rational numbers.}
\end{enumerate}

\end{corollary}

Examples for the above two cases will be presented  later in sections \ref{subsec: non-real} and \ref{ex:pipsi}.

\subsection{Translationally equivariant minimal Lagrangian tori}
\label{subsec: equiv tori}

\subsubsection{Basic discussion of possible tori}

By definition, a minimal Lagrangian torus $\mathbb{T}$ is a minimal Lagrangian surface
which admit for some  $\lambda = \lambda_0$ two over $\R$ linearly independent  periods  
$\omega_1 = p_1+ im_12 T$ and $\omega_2 = p_2+ i m_2 2 T,$
with real numbers $p_1,p_2$ and integers $m_1,m_2$.
Hence $\mathbb{T}$ is of the form $\mathbb{T}= \C/\mathcal{L}$, where $\mathcal{L}$ is a rank 2 lattice. 
Then $ \hat{p} = m_2 \omega_1 - m_1 \omega_2 \in \mathcal{L}$  is a real period of $f$. Since $\omega_1$ and $\omega_2$ are linearly independent, it follows that $\hat{p}$ is not $0$, i.e., $\hat{p}$ is a nonzero real period of $f$.
Therefore,  by Corollary \ref{specialperiods} we obtain that
$ r(\lambda_0) = d_1(\lambda_0) / d_2(\lambda_0)$ is a rational number.
Thus every translationally equivariant minimal Lagrangian torus admits a real period and a non-real period.

Next we consider the period lattice $$\mathcal{L}(f) = \lbrace p+m2Ti \in \C; f(z+p+m2Ti) = f(z)
\mbox{ for all } z \in \C\rbrace$$
associated with a translationally equivariant minimal Lagrangian surface $f$.
Note that $\mathcal{L}(f)$ is indeed a lattice. 

For a general minimal Lagrangian surface the period lattice will be empty. For some such surfaces it will be of the form $\omega \mathbb{Z}$.  Our goal in this section is to understand better the case where the period lattice is a lattice of rank 2.
Clearly, if $\mathbb{T} = \C/ \mathcal{L}$ is a  translationally equivariant minimal Lagrangian torus, then $\mathcal{L} \subset \mathcal{L}(f)$ holds
and also $\mathbb{T}(f) = \C/ \mathcal{L}(f) $ is a translationally equivariant minimal Lagrangian torus.

More precisely,

\begin{proposition}
Assume the translationally equivariant minimal Lagrangian surface $f$ defined on $\C$ descends to some torus $\hat{\mathbb{T}}$, then this torus is induced by some sub-lattice $\hat{\mathcal{L}}$ of $\mathcal{L}$ and there exists a covering $\hat{\pi}: \hat{\mathbb{T}} \rightarrow \mathbb{T}$ with fiber $\mathcal{L}/\hat{\mathcal{L}}$. In particular, if $f$ descends to some torus, it can be injective only if the torus is the one defined by the period lattice.
In particular, an embedding of a translationally equivariant minimal Lagrangian torus is only possible, if the torus is defined by the period lattice.
\end{proposition}

\subsubsection{The period lattice}

In the case under consideration it is fortunately possible to give a fairly precise description of the period lattice.

\begin{theorem} \label{periodlattice}
The period lattice $\mathcal{L}(f)$ of any translationally equivariant minimal Lagrangian torus $f$ is of the form
\begin{equation*}
 \mathcal{L}(f) = p_f \mathbb{Z} + \omega_f \mathbb{Z},
\end{equation*}
where $p_f$ is the smallest (real) positive period and $\omega_f$ the period with smallest positive imaginary part.
\end{theorem}
\begin{proof}
We have seen above that any translationally equivariant minimal Lagrangian torus has a non-zero real period. Let $p_f$ denote the smallest positive real period of $f$. Assume $p$ is any other positive period. Then $0 < p_f < p$.
If $p$ is not an integer multiple of $p_f$, then we can substract an integer multiple from $p$ such that $0 < p - k p_f < p_f$. This is a contradiction.
Let's consider next all non-real periods of $f$ and let's choose any such period $\omega_f = q + m_f i 2T$ for which $m_f$ is positive and minimal.
Now choose any other period $\omega = a + b i2T$, with $a \in \R$ and $b$ an integer. We can assume that $b$ is positive. If $b$ is not an integer multiple of $m_f$, then one can subtract an integer multiple of 
$\omega_f$ from $\omega$   
such that $\omega - k \omega_f = (a-kq)  +(b - km_f)i 2T$  and $0 < b-km_f < m_f$. This is a contradiction. Therefore $b = m m_f$ with an integer $m$.
 
Moreover $\omega - m \omega_f = a - m q$ is a real period. But we have seen above that all real periods are an integer multiple of $p_f$. Hence $a-m q = np_f$ and $\omega = np_f + m\omega_f$ follows.

\end{proof}

Since the two generating periods for the period lattice $\mathcal{L}$ above
are determined by some minimality condition, to find all  translationally equivariant minimal Lagrangian tori it basically suffices to find a real period and a non-real period. The existence of such periods can be rephrased as follows

\begin{theorem}\label{torus condition}
Let $f$ be a  translationally equivariant minimal Lagrangian immersion.
Then $f$ descends to a torus if and only if
\begin{enumerate}
\item[(1)] The eigenvalues  $d_1(\lambda_0)$ and $d_2(\lambda_0)$ have a rational quotient.

\item[(2)] Either the eigenvalues  $d_1(\lambda_0)$ and $d_2(\lambda_0)$ 
equal or\\ 
$\frac{1}{2\pi}\mathrm{Im}\beta_2(2T, \lambda_0)\lbrace d_1(\lambda_0)d_2(\lambda_0)+\frac{2\beta}{3}\rbrace $ is rational.
\end{enumerate}
\end{theorem}
\begin{proof}
We know that (1) is equivalent with the existence of a real period and $(2)$ follows for a non-real period by (\ref{periodicity-cyl4}). It thus remains to show that $(1)$ and $(2)$ together imply the existence of a non-real period.
First, if $d_1$ and $d_2$ are equal (for a fixed $\lambda = \lambda_0$), then
(\ref{mLi-cyl3}) actually is only one equation and one can compute $p$ for $m=1$. 
Actually we see from Lemma \ref{D-roots} that $D(\lambda)$ having multiple eigenvalues implies that the minimal Lagrangian surface is flat and needs to be a part of the Clifford torus.
 Assume now $d_1 \neq d_2$. Then we can compute $m \neq 0,$ $l_1$ and $l_2$ from the first equation in (\ref{periodicity-cyl4})
and then $p$ from the second equation in (\ref{periodicity-cyl4}).

\end{proof}

\subsubsection{The case of a real cubic form $\lambda^{-3}\psi$}

We know that in the case  of a real cubic form $\lambda^{-3}\psi$, the canonical lift $F$ is invariant under translations by $\omega = 4Ti$ (see Section \ref{ex:pipsi}).  

From Theorem \ref{periodlattice} we know that in the case under consideration
the period lattice is spanned by a real period and a non-real period with smallest positive imaginary part. This non-real period is thus either $4Ti$ or
of the form $b+m2Ti$ with $m = 2k +1$. Then we can assume that this second period is of the form $b + 2Ti$.

Moreover, with $\omega$ 
also $2\omega = 2b + 4Ti$ is a period, whence $2b$ is a period.
Since we can assume that either $b=0$ or $0 < b < p_f$, we obtain 
$b= \frac{1}{2} p_f$.

At any rate,  the quotient of the eigenvalues $d_1$ and $d_2$ of $D$ is rational.

\begin{proposition}
Keeping the definitions and the notation introduced for translationally equivariant minimal Lagrangian surfaces we obtain in the case of a real cubic form $\lambda^{-3}\psi$ 
the possible period lattices $\mathcal{L}(f)  = p_f\mathbb{Z} + 4Ti\mathbb{Z}$
and $\mathcal{L}(f) =p_f \mathbb{Z}+ (\frac{1}{2}p_f + 2Ti)\mathbb{Z}.$
In both cases, the quotient of the eigenvalues $d_1$ and $d_2$ of $D$ is rational.
Conversely, if the cubic form $\lambda^{-3}\psi$ is real and $d_1/d_2$ is rational, then the corresponding translationally equivariant minimal Lagrangian surface descends to some torus which is  defined by a lattice of the type given above.
\end{proposition}

\section{Comparison with the work of  Castro-Urbano}

In this section, we will show how our approach relates to the one of  Castro-Urbano \cite{CU94}.
As before, also in this section we will consider the whole associated family.

To simplify notation, in this subsection we will (usually) not indicate dependence on variables like $z$, $\bar{z}$ or $\lambda$.

Let again $f :\C \rightarrow \C P^2$ denote  the associated family of  translationally equivariant 
minimal Lagrangian immersions with horizontal conformal lift $F$ and frame $\F$.

Then $f$ is generated by some matrix $D$ and 
$$\mathbb{F}(x,y)=e^{xD}\mathbb{F}(y)$$
holds.

The characteristic polynomial of $D$ is given by \eqref{eq:charpol},
therefore we immediately obtain
\begin{equation} \label{eq:cubic}
\partial_x^3 \F +\beta \partial_x \F+\mu_0 \F=0.
\end{equation}

\begin{remark}
We would like to point out that instead of using the third order  equation
above, in \cite{CU94} the authors prove the existence of a sixth order equation due to the real orthogonal frames they used. 
So from here on our computations are usually somewhat simpler, but follow a very similar idea.
\end{remark}

Equation  \eqref{eq:cubic} holds, of course, for each column of $\F$ separately. 
In particular, we know that the immersion $f(x,y)$ is given by
$$f(x,y)=[\mathbb{F}(x,y) e_3]=[e^{xD}\mathbb{F}(y)e_3].$$

Thus the horizontal conformal lift $F(x,y) = \F (x,y) e_3$ satisfies 

\begin{equation}\label{eq:cubicF}
F(x,y) = \exp(xD) F(y)
\end{equation}
and therefore also
\begin{equation*}
\partial_x^3 F +\beta \partial_x F- 2i \mathrm{Re}(\lambda^{-3}\psi) F=0.
\end{equation*}

Recall that for a minimal Lagrangian immersion  
$f: M\rightarrow \mathbb{C}P^2$ with induced metric $g=2e^udzd\bar{z}$, its horizontal 
lift $F: U\rightarrow S^5(1)\subset \mathbb{C}^3$ satisfying the  equations \eqref{conformal}, \eqref{eq:frame1}, \eqref{eq:UV1}
with $\psi$ defined by \eqref{eq:phipsi} gives an associated family of minimal Lagrangian surfaces with the cubic differential $-i\lambda^{-3}\psi$. Explicitly, the associated extended frame $F(z,\bar{z}, \lambda)$ satisfies
\begin{equation}\label{eq:frameCU}
\begin{split}
F_{zz}&=u_zF_z-e^{-u}i\lambda^{-3}\psi F_{\bar z},\\
F_{z\bar z}&=-e^u F,\\
F_{\bar{z}\bar{z}}&=-e^{-u}i\lambda^{3}\bar{\psi}F_{z}+u_{\bar z}F_{\bar z}.
\end{split}
\end{equation}

It is straightforward to rewrite the equations  \eqref{eq:frameCU} 
 involving derivatives for $z$ and $\bar z$ into
\begin{flalign}
 & F_{xx}=-ie^{-u}\mathrm{Re}(\lambda^{-3}\psi)F_x -\frac{u^{\prime}-2i e^{-u}\mathrm{Im}(\lambda^{-3}\psi)}{2}F_y -2e^u F, &\label{eq:realCU01}\\
&F_{xy}=\frac{u^{\prime}+2i e^{-u} \mathrm{Im}(\lambda^{-3}\psi)}{2} F_x+i e^{-u} \mathrm{Re}(\lambda^{-3}\psi)F_y, &\label{eq:realCU02}\\
&F_{yy}=i e^{-u} \mathrm{Re}(\lambda^{-3}\psi) F_x+\frac{u^{\prime}- 2 i e^{-u} \mathrm{Im}(\lambda^{-3}\psi)}{2} F_y-2e^u F. &\label{eq:realCU03}
\end{flalign}

We want to evaluate \eqref{eq:cubicF} by writing $F$ as a linear combination of eigenvectors of $D$.

It follows from \eqref{eq:charpol} that the eigenvalues $\mu_1=id_1$, $\mu_2=id_2$, $\mu_3=id_3$ satisfy
\begin{equation*}\label{eq:lambda_relation}
d_1+d_2+d_3=0, \quad d_1 d_2+d_2 d_3+d_3 d_1=-\beta,
\quad d_1 d_2 d_3=-2\mathrm{Re}(\lambda^{-3}\psi).
\end{equation*}

Let $l_1, l_2$ and $l_3$ denote an  orthonormal basis of  eigenvectors of $D(\lambda)$ for the eigenvalues 
$i d_1, \, id_2, \, id_3$, respectively.
Then there exist scalar functions $p_j(y)$ such that
\begin{equation*}
F(y)=p_1(y) l_1+p_2(y)l_2+p_3(y)l_3
\end{equation*}
holds.
As a consequence, for  $F(x,y) = \exp(xD) F(y)$ we obtain 
\begin{equation}\label{eq:decF}
F(x,y)=p_1(y)e^{i d_1 x}l_1+p_2(y)e^{i d_2 x}l_2+p_3(y) e^{i d_3 x}l_3.
\end{equation}

Next we evaluate equation \eqref{eq:realCU02} and obtain for $j = 1,2,3$ the scalar equations
\begin{equation}\label{eq:66'}
(d_j e^u -\mathrm{Re}(\lambda^{-3}\psi) )p_j^{\prime}= \frac{u^{\prime} e^u +2i \mathrm{Im}(\lambda^{-3}\psi)}{2}d_j p_j.
\end{equation}

\begin{remark}
\begin{enumerate}
\item  The equations \eqref{eq:frameCU} lead to three real differential equations  and  via \eqref{eq:decF} yield three scalar differential equations for the coefficient functions $p_j(y)$.
Two of these three differential equations are of first order and of the form
$A_j p_j^\prime = B_j p_j$ and the third one is a second order equation with leading coefficient $1$.
Since the two first order equations describe the same function $p_j$ we obtain for the equivalence of these two equations the  identity $A_1 B_2 = A_2 B_1$ which turns out to be 
\begin{equation}\label{eq:factor-nonzero}
[d_j e^u-\mathrm{Re}(\lambda^{-3}\psi)][d_j^2 e^u +\mathrm{Re}(\lambda^{-3}\psi)d_j-2e^{2u}]
=[\frac{1}{4} (u^{\prime})^2e^{2u}+(\mathrm{Im}(\lambda^{-3}\psi))^2] d_j.
\end{equation}

\item There are several cases that need to be distinguished:
\begin{enumerate}
\item The first case is, where the matrix $D(\lambda_0) $ is not invertible. In this case $\lambda_0^{-3}\psi$ is purely imaginary and one eigenvalue vanishes, say $i d_1(\lambda_0) = 0$, and the other two eigenvalues are 
$id_\pm(\lambda_0) = \pm i \sqrt{\beta}$.
This case will be discussed separately. Therefore, in the rest of this remark we will always assume that all eigenvalues are non-zero at all values of $\lambda$ considered.

\item  Assuming now that no eigenvalue $d_j(\lambda_0)$ vanishes, it can happen that two eigenvalues coalesce. In this case we know from Lemma \ref{D-roots}
that the minimal Lagrangian surface is flat, a case which is no longer considered at this point. Therefore, from now on we will assume that all eigenvalues are different and non-zero at all values of $\lambda$ considered.
\end{enumerate}

\item There are two more cases to distinguish. Namely the cases where $\lambda_0^{-3}\psi$ is real and non-real and non-purely-imaginary.  
These two cases will also be treated separately below.

\item  In view of \eqref{eq:factor-nonzero}
it turns out to be useful to note  that if  $d_j(\lambda_0) \neq 0$, and if  $d_j (\lambda_0) e^{u(y_0)} -\mathrm{Re}(\lambda_0^{-3}\psi) = 0$,  then $\lambda_0^{-3}\psi$ is real and $u^\prime (y_0) = 0$.

\item One could evaluate the  remaining two equations \eqref{eq:realCU01} and \eqref{eq:realCU03} 
in an analogous manner. However, it turns out that these two equations do not produce any new information,
if  $d_j(\lambda_0) \neq 0$ and $d_j (\lambda_0)e^u -\mathrm{Re}(\lambda_0^{-3}\psi) \neq 0$.

\end{enumerate}
\end{remark}

\subsection{The case of non-invertible $D(\lambda)$}

In view of \eqref{eq:charpol} it is clear that 
$\mathrm{Re}(\lambda_0^{-3}\psi)=0$ is equivalent to that $D(\lambda_0)$ is not invertible and to $\lambda_0^{-3}\psi$ being purely imaginary.  

Let's assume now that $d_1(\lambda_0) = 0$.  Then, fixing $\lambda = \lambda_0$,  the eigenvalues of $D$ are, without loss of generality, $id_1 = 0, id_2 = i \sqrt{\beta}$ and  
$id_3 =  -i\sqrt{\beta}$. 

We note that, in full generality, the equation \eqref{eq:realCU01} translates, in view of \eqref{eq:decF},  to 
\begin{equation} \label{eq:caseFxx}
( u^\prime e^u - 2 i \mathrm{Im}(\lambda^{-3} \psi)) p_j^\prime = 2( d_j^2 e^u + \mathrm{Re}(\lambda^{-3}\psi)d_j - 2 e^{2u}) p_j.
\end{equation}

Note that here the coefficient of $p_j^{\prime}$ on the left side does not vanish in the case under consideration, where $\mathrm{Re}(\lambda^{-3}\psi) =0$.)
Writing out  the three equations of \eqref{eq:caseFxx} it is easy to observe that the differential equations for $p_2$ and $p_3$ are equal. Therefore, the solutions $p_2$ and $p_3$ of these differential equations only differ by some constant. But then, say $p_3 = \alpha p_2$, we obtain $|\alpha| = 1$, since $F$ has length $1$. As a consequence, up to some isometry of $\C P^2$ the surface only takes value in some hyperplane. This is a case we are not interested in.

\subsection{The case of non-real $\lambda^{-3}\psi$} 
\label{subsec: non-real}

Now let's assume that $\lambda^{-3}\psi$ is not real. Then $d_j e^u -\mathrm{Re}(\lambda^{-3}\psi)\neq 0$.
We obtain
\begin{equation*}\label{sol_pj}
p_j(y)= \rho_j (d_je^u-\mathrm{Re}(\lambda^{-3}\psi))^{\frac{1}{2}} e^{i\int_0^y \frac{d_j \mathrm{Im}(\lambda^{-3}\psi)}{d_j e^u-\mathrm{Re}(\lambda^{-3}\psi)}ds},
\end{equation*}
where $\rho_1, \rho_2, \rho_3$ are independent of $z$.

To determine the coefficients $\rho_j$, we recall that the lift $F$ is conformal and horizontal, whence we have
\begin{equation*}\label{eq:productFxFy}
\begin{split}
& F\cdot \bar{F}=1,\,\, F_x\cdot \bar{F}=F_y\cdot \bar{F}=F_x\cdot \overline{F_y}=0,\\
& F_x\cdot\overline{F_x}=F_y\cdot \overline{F_y}=2e^u.
\end{split}
\end{equation*}
These equations lead to the following $3$ equations:

\begin{enumerate}
\item $\sum_{j=1}^3 |\rho_j|^2(d_je^u-\mathrm{Re} (\lambda^{-3}\psi))=1,$\\
\item $\sum_{j=1}^3 d_j |\rho_j|^2(d_je^u-\mathrm{Re} (\lambda^{-3}\psi))=0,$\\
\item $\sum_{j=1}^3 d_j^2  |\rho_j|^2(d_je^u-\mathrm{Re} (\lambda^{-3}\psi))=2e^u.$
\end{enumerate}

Since the Vandermonde matrix built from distinct $d_1, d_2,d_3$ is  
invertible, if the surface is not flat
(see Lemma \ref{D-roots}), this system of equations can be solved for $ |\rho_j|^2 (d_j e^u - \mathrm{Re}(\lambda^{-3}\psi))$.

Since  $d_je^u-\mathrm{Re}(\lambda^{-3}\psi)\neq 0$,   we obtain 
$$|\rho_j|^2=\frac{1}{d_j^3-\mathrm{Re}(\lambda^{-3}\psi)}, \quad j=1,2,3.$$
Choose $\rho_j=\frac{1}{(d_j^3-\mathrm{Re}(\lambda^{-3}\psi))^{\frac{1}{2}}}$, 
then we have
\begin{equation}\label{eq:F(x,y,lambda)}
F(x,y, \lambda)= h_1(y)e^{id_1 x +iG_1(y)} l_1 +   h_2(y)e^{id_2 x +iG_2(y)} l_2 + h_3(y)e^{id_3 x +iG_3(y)}l_3,
\end{equation}
where 
\begin{equation}\label{eq:h_j G_j}
h_j(y)=\left(\frac{d_je^u-\mathrm{Re}(\lambda^{-3}\psi)}{d_j^3-\mathrm{Re}(\lambda^{-3}\psi)}\right)^{\frac{1}{2}},
\quad\quad
G_j(y)=\int_0^y \frac{d_j \mathrm{Im}(\lambda^{-3}\psi)}{d_j e^u-\mathrm{Re}(\lambda^{-3}\psi)}ds.
\end{equation}
Note that also the eigenvalues $d_j$ depend on $\lambda$.
In terms of the orthonormal basis of eigenvectors of $D(\lambda)$ chosen above, we can assume, 
by the discussion just above, that the phase factor of the $l_j$  is chosen such that $h_j$ is positive and real. 
We will therefore continue to denote this basis by the  letters $l_1,l_2$ and $l_3$.

Next we want to consider $F(x+p, y+ m2T, \lambda)$.
At one hand we obtain
\begin{equation} \label{equ1}
   F(x+p,y+m2T,\lambda)=\sum_{j=1}^3 h_j(y + m2T,\lambda)e^{id_j(\lambda) (x+p) +iG_j(y + m2T,\lambda))}l_j(\lambda),
\end{equation}
and on the other hand we obtain
\begin{equation*} \label{equ2}
F(x+p,y+m2T,\lambda) = M(\lambda) F(x,y,\lambda).
\end{equation*}

Using the simple equations $h_j(y + m2T,\lambda) = h_j(y,\lambda)$, 
since $e^u$ is $2T-$periodic, and  the obvious identity $G_j(y + m2T,\lambda) = G_j(y,\lambda)+m G_j(2T,\lambda)$,
we see that the coefficient for $l_j(\lambda)$ in the equation (\ref{equ1}) actually is of the form 
\begin{equation}\label{equ3}
e^{id_j(\lambda)p + i mG_j(2T,\lambda))} \cdot h_j(y,\lambda) e^{i(d_j(\lambda)x + G_j(y,\lambda))}.
\end{equation} 

Since the $l_j(\lambda)$ are eigenvectors of $M(\lambda)$,
the left factors of these expressions in \eqref{equ3} 
are exactly the eigenvalues of $M(\lambda)$.
Hence comparing with \eqref{eig-mono}, we obtain 

\begin{theorem}
Retaining the notation used so far we obtain for every translation $p + m2Ti$ and $j = 1,2,3$ the equation
\begin{equation*}
\begin{split}
p d_j(\lambda) + mG_j(2T,\lambda) \equiv
pd_j(\lambda)-m\lbrack\mathrm{Re}\beta_1(2T, \lambda) d_j(\lambda)\\
+\mathrm{Im}\beta_2(2T,\lambda)(-d_j(\lambda)^2+\frac{2\beta}{3})\rbrack \mod 2 \pi \mathbb{Z}.
\end{split}
\end{equation*}
\end{theorem}

Actually we can directly show that $G_j(2T,\lambda)+\mathrm{Re}\beta_1(2T, \lambda) d_j(\lambda)
+\mathrm{Im}\beta_2(2T,\lambda)[-d_j(\lambda)^2+\frac{2\beta}{3}]=0$ for each $j$.
Note that by summing up the three equations we obtain 
\begin{equation}\label{G_j sum}
G_1(2T,\lambda) + G_2(2T, \lambda) + G_3(2T,\lambda) = 0 
\end{equation}
for all $\lambda$.

From the argument in Section \ref{subsec: equiv tori}, if $f$ descends to a torus, then
 $d_1/d_2$ is rational. Moreover, if there exists some $p\in \mathbb{R}$ and $m\in \mathbb{Z}$ such that $f(x+p, y+m2T)=f(x,y)$, then
$$p d_j(\lambda)+m G_j(2T,\lambda)\in 2\pi \mathbb{Z}$$
for $j=1,2,3$.
From $d_1+d_2+d_3=0$ and \eqref{G_j sum}, we can easily obtain that
$$\frac{1}{2\pi} (\frac{d_2}{d_1} G_1(2T,\lambda)-G_2(2T,\lambda))$$ is rational. The converse obviously also holds.
So we have

\begin{theorem}[\cite{CU94}]
If the cubic form $\lambda^{-3}\psi$ of a translationally equivariant  minimal Lagrangian surface $f$ is not real, then the canonical horizontal lift $F$ of $f$ has the form of \eqref{eq:F(x,y,lambda)}.
In this case, $f$ descends to a torus if and only if both $d_1/d_2$ and
$\frac{1}{2\pi} (\frac{d_2}{d_1} G_1(2T)-G_2(2T))$ are rational. 
\end{theorem}

\subsection{The case of real $\lambda^{-3}\psi$}
\label{ex:pipsi}

Retaining the notation used so far,  we assume in this subsection that $\lambda^{-3}\psi$ is real.
In this case, $d_j e^u -\mathrm{Re}(\lambda^{-3}\psi)$ can vanish at some points.

Recall that $a_1, a_2, -a_3$ are the roots of $w^3-\frac{\beta}{2}w^2+\frac{|\psi|^2}{2}=0$ in Section \ref{subsec:metric}.
We see that the roots of 
the characteristic polynomial \eqref{eq:charpol} of $D(\lambda)$ 
can be given by
\begin{equation*}
\mu_1=id_1=\frac{i \lambda^{-3}\psi}{a_1}, \quad \mu_2=id_2=\frac{i \lambda^{-3}\psi}{a_2} , \quad 
\mu_3=id_3=-\frac{i \lambda^{-3}\psi}{a_3}.
\end{equation*}

Recall the following properties of the Jacobi elliptic functions
\begin{eqnarray*}
&&\sn^2 z+\cn^2 z=1, \, k^2 \sn^2z+\dn^2 z=1,\\
&&\frac{d}{dz} \sn z=\cn z \dn z, \, \frac{d}{dz} \cn z=-\sn z \dn z,\, \frac{d}{dz} \dn z=-k^2 \sn z \cn z. 
\end{eqnarray*}
Taking into account formulas \eqref{eq:equiv_u0}, \eqref{eq:k}, \eqref{eq:q&r}, we can rewrite \eqref{eq:66'} as
\begin{equation*}
\begin{split}
 p_1^{\prime}(y)\sn(ry, k)&=p_1(y) \frac{d}{dy}\sn(ry, k),\\ 
 p_2^{\prime}(y)\cn(ry, k)&=p_2(y) \frac{d}{dy}\cn(ry, k),\\
  p_3^{\prime}(y)\dn(ry, k)&=p_3(y) \frac{d}{dy}\dn(ry, k).
\end{split}
\end{equation*}
Integration gives $$p_1(y)=c_1 \sn(ry, k), \, p_2(y)=c_2 \cn(ry, k), \, p_3(y)=c_3 \dn(ry, k),$$
where $c_1$, $c_2$ and $c_3$ are constant complex numbers.

Thus by using an analogous argument involving the Vandermonde matrix given in Subsection \ref{subsec: non-real},  we obtain the constants $c_1, c_2, c_3$ from \eqref{eq:h_j G_j}: 
\begin{equation*}\label{eq:c}
c_1=a_1 \sqrt{\frac{a_1-a_2}{a_1^3-|\psi|^2}},\quad
c_2= a_2\sqrt{\frac{a_1-a_2}{|\psi|^2-a_2^3}},\quad
c_3= a_3\sqrt{\frac{a_1+a_3}{|\psi|^2+a_3^3}}.
\end{equation*}
This yields the canonical horizontal lift of the associated minimal Lagrangian surface
\begin{equation}\label{eq:F_psi II}
F(x,y,\lambda)=c_1 \sn(ry, k) e^{id_1x} l_1 +  c_2 \cn(ry, k) e^{id_2x}l_2 + c_3 \dn(ry, k) e^{id_3x} l_3, 
\end{equation}
where 
$l_1, l_2, l_3$ is an orthonormal basis of  eigenvectors of $D$ for the eigenvalues 
$\mu_1, \, \mu_2, \, \mu_3$, respectively.

Based on our discussion in Section \ref{Sec:equiv cylinder tori}, we immediately obtain 
\begin{theorem}
If the cubic form $\lambda^{-3}\psi$ of a translationally equivariant  minimal Lagrangian surface $f$ is real, then the canonical horizontal lift $F$ of $f$ has the form
of \eqref{eq:F_psi II}
and satisfies $F(x, y+4T)=F(x,y)$.
In particular, $f$ is defined on the cylinder $ \mathcal{C} = \C / 4Ti \mathbb{Z}$.

If there also exists some $\tau\in \mathbb{R}$ such that $f(x+\tau, y)=f(x,y)$, then
$e^{id_1 \tau}=e^{id_2 \tau}=e^{id_3\tau}$, which implies  $d_1/d_2$ is rational.
In this case, $f$ descends to the torus $ \mathbb{T} = \C/\mathcal{L}$, where  $\mathcal{L}$
denotes the lattice $\mathcal{L} = 4Ti \mathbb{Z} + \tau \mathbb{Z}$.

Conversely, if $d_1/d_2$ is rational, then there exists some $\tau \in \R$ such that $f(x+\tau, y)=f(x,y)$ holds.
\end{theorem}

\begin{acknow}
This work has been done during the first author's visits at Tsinghua University and the second author's visit at TU-M\"{u}nchen. The authors would like to thank both institutions for their generous support.
Most of the results of this paper were reported by the second author during the 10th geometry conference for the friendship between China and Japan in 2014. She would like to thank the organizers for the invitation to the conference. This work is partially supported by NSFC grant No.~11271213.
\end{acknow}


\end{document}